\documentclass [smallextended,final]{svjour3}
\usepackage[T1]{fontenc}
\usepackage[utf8]{inputenc}
\usepackage{newtxtext,newtxmath}
\smartqed
\usepackage{multirow}
\usepackage{pgfplots, pgfplotstable, tikz}
\usetikzlibrary{positioning}
\usepackage{amsmath,amsfonts,mathtools}
\usepackage{graphicx}
\usepackage{subcaption}
\usepackage{derivative}
\usepackage{algorithm}
\usepackage{csvsimple}
\usepackage{algpseudocode}
\usepackage{listing}
\usepackage{hyperref}
\usepackage{hypcap}
\usepackage{multirow}
\usepackage{hhline}
\usepackage{colortbl}
\usepackage{diagbox} 
\usepackage{makecell}

\usepackage{subdepth}

\newcommand{\mc}{\mathrm{i}}
\newcommand{\be}{\begin{equation}}
\newcommand{\ee}{\end{equation}}
\newcommand{\ba}{\begin{array}}
\newcommand{\ea}{\end{array}}

\newcommand{\comment}[1]{}

\newcommand{\qedwhite}{\hfill \ensuremath{\Box}}
\pgfplotstableset{col sep=comma}
\pgfplotsset{compat=1.16}
\newcommand{\B}[1]{\mbox{\boldmath $#1$}}
\newcommand\smallO{
  \mathchoice
    {{\scriptstyle\mathcal{O}}}
    {{\scriptstyle\mathcal{O}}}
    {{\scriptscriptstyle\mathcal{O}}}
    {\scalebox{.7}{$\scriptscriptstyle\mathcal{O}$}}
  }

\begin{document}

\title{Computing the Action of the Generating Function of Bernoulli Polynomials on a  Matrix with An Application to Non-local Boundary Value Problems}
\titlerunning{Computing the Action of the Matrix  Generating Function of Bernoulli Polynomials}       

\author{Lidia Aceto \and  Luca Gemignani}

\institute{L. Aceto \at Dipartimento di Scienze e Innovazione Tecnologica\\
  Universit\`a del Piemonte Orientale \\
              \email{lidia.aceto@uniupo.it}   \\
            https://orcid.org/0000-0002-4537-2444
           \and
           L. Gemignani \at Dipartimento di Informatica\\
  Universit\`a di Pisa \\
              \email{luca.gemignani@unipi.it} \\
        https://orcid.org/0000-0001-8000-4906
}

\date{Received: date / Accepted: date}

\maketitle
\begin{abstract}
This paper deals with efficient numerical methods for computing the action of the generating function of Bernoulli polynomials, say $q(\tau,w)$, on a typically large sparse matrix. This problem occurs when solving some non-local boundary value problems. Methods based on the Fourier expansion of $q(\tau,w)$ have already been addressed in the scientific literature. The contribution of this paper is twofold. First, we place these methods in the classical framework of Krylov-Lanczos (polynomial-rational) techniques for accelerating Fourier series. This allows us to apply the convergence results developed in this context to our function. Second, we design a new acceleration scheme. Some numerical results are presented to show the effectiveness of the proposed algorithms.

\end{abstract}

\section{Introduction}
In this work we are interested in the approximation of the function
\begin{equation} \label{eq:init}
 q(\tau,w)= \frac{w e^{w\tau}}{e^w -1}, \qquad w \in \mathbb{C} \setminus \{0, \pm 2\pi \mc, \pm 4 \pi \mc, \ldots\},
\end{equation}
which acts on a typically large sparse matrix.
As known, see e.g. \cite{Abramowitz}, this function is the generating function of the Bernoulli polynomials $B_k(\tau),$ i.e.
\begin{equation}\label{formalseries}
q(\tau,w) = \sum_{k=0}^{+\infty }B_k(\tau) \frac{w^k}{k!},
\end{equation}
with this formal series which converges uniformly only  within  $|w|<2\pi$. 
Bernoulli polynomials  are very important in various fields of mathematics and, in particular, are intimately connected to differential equations. For example, recall that in a finite-dimensional setting in \cite{Boito1} it is proved that for a given matrix $A\in \mathbb R^{s\times s}$ and a given vector $\B {f} \in \mathbb{ R } ^s,$ the non-local boundary value problem (BVP)
\begin{eqnarray}
    &&\odv{\mathbf{u}}{\tau} = A \mathbf{u}, \quad 0<\tau<1,   \label{l1} \\
  &&\displaystyle \int_0^1 \mathbf{u}(\tau) \,\mathrm{d}\tau = \B {f}, \label{l2}
\end{eqnarray}
admits as unique solution
\begin{equation}\label{l3}
 \mathbf{u}(\tau)= q(\tau,A) \B {f}. 
 \end{equation}
More general differential problems can  also be reduced into the form \eqref{l1}, \eqref{l2} by the application of semidiscretization-in-space schemes.

BVPs  with special  non-local  boundary conditions are ubiquitous in applied sciences and engineering.  Integral boundary conditions are encountered in various applications such as population
dynamics, blood flow models, chemical engineering and
cellular systems (see e.g. \cite{Bou,Mar}  and the references given therein).  Interestingly, such problems  are also relevant in control theory for system identification, where  the analysis goes back to the work of Bellmann \cite{Bell}. 

It should be noted that functional approximations of \eqref{eq:init} based on \eqref{formalseries} are unreliable due to the restrictions on the radius of convergence. Recently, this drawback has been overcome by the functional method described in \cite{Boito1,Boito}. By applying an ad hoc technique for the convergence acceleration of a Fourier series, the following rational-trigonometric modification of \eqref{formalseries} was obtained in \cite{Boito} for $n\ge 0$
\begin{equation}\label{lexp}
 q(\tau,2 \pi z)= \sum_{k=0}^{2n+1} B_{k} (\tau) \frac{(2\pi z)^{k}}{k!} + (-1)^{n} 2 z^{2n+2}\sum_{k=1}^{+\infty} \frac{ \cos(2\pi k \tau) + \displaystyle\frac{z}{k} \sin(2\pi k \tau)}{k^{2n} (z^2+k^2)}.
 \end{equation}
Hence, by truncating the series on the right-hand side to a finite number of terms, say $N,$ 
in the same paper it was proposed to consider the family of approximations given by
\begin{equation}\label{lexpN}
\!\!\! g_{n,N}(\tau,z):= \sum_{k=0}^{2n+1} B_{k} (\tau) \frac{(2\pi z)^{k}}{k!} + (-1)^{n} 2 z^{2n+2}\sum_{k=1}^N \frac{ \cos(2\pi k \tau) + \displaystyle\frac{z}{k} \sin(2\pi k \tau)}{k^{2n} (z^2+k^2)}.
\end{equation}
While Arnoldi-type methods can also be used to evaluate the matrix function $q(\tau,A)$, this functional approximation method offers the following key advantages over the usual Arnoldi reductions:
\begin{itemize}
\item[$\bullet$] The functional approximation has the form of a series expansion in the time coordinate with coefficients that only depend on $A$ and $\B f$. In this way, these coefficients can be precomputed once and reused for the tabulation of the function $\mathbf{u}(\tau)$ over any grid of points. 
\item[$\bullet$] Although the function \eqref{eq:init} is meromorphic, the considered functional approximation exploits the pole structure of this function by providing the holomorphic extension of
\eqref{eq:init} around the removable singularity at the origin of the complex plane. This approach proves to be very suitable for dealing with matrices whose eigenvalues are clustered near the origin, as a numerical example in Section~\ref{numexp} will show.
\item[$\bullet$] The method can easily be extended to handle the case where $A$ is a linear unbounded closed operator between two Banach spaces \cite{Boito3}. This allows the method to be adapted to the treatment of certain parabolic differential problems, which can be recast into the form \eqref{l1}-\eqref{l2}.
\end{itemize}

Although it has also been shown that the series on the right side of \eqref{lexp}  converges uniformly to $q(\tau, 2\pi z)$  on every set  $\mathcal K_1\times \mathcal K_2,$  with   $\mathcal K_1\subset (0, 1)$ and $\mathcal K_2\subset \mathbb C\setminus \pm   \mc \mathbb N$  compact sets, both the residual estimates and the exact evaluations of the  asymptotic error constants are still missing.
Furthermore, although increasing the value of $n$ can bring great benefits to improve the convergence rate of \eqref{lexp}, computing the polynomial term of $g_{n, N}(\tau, A)$ can be prone to numerical instabilities whenever $A$ has eigenvalues of large magnitude \cite{Boito3}.

This work aims to fill these gaps.  The starting point is the derivation of the formula \eqref{lexp} in a more general framework for the development of acceleration methods for ordinary Fourier series, which in \cite{Pogh} is called the Krylov-Lanczos approximation, since this approach was originally proposed by Krylov \cite{Krylov} and Lanczos \cite{Lanczos}.
The term Krylov-Lanczos approximation denotes a class of methods for accelerating trigonometric series of a non-periodic function by polynomial corrections that represent discontinuities of the function and some of its first derivatives.
A classical approach among these methods is the Lanczos representation of functions in terms of Bernoulli polynomials and trigonometric series \cite{Lanczos,Lyn}. It is shown that this approach, applied to the periodic extension of $q(\tau, w)$, leads to the expansion \eqref{lexp}. The result allows us to adapt the techniques and tools developed in the general framework for the analysis of the function $q(\tau,w)$. Specifically, we derive estimates of the residual error in $L^2$-norm and for pointwise convergence  in the regions 
$\mathcal K_1$
far from the discontinuities at the endpoints.  Furthermore,  we develop an alternative acceleration  strategy  for $\tau\in \mathcal K_1$
which replaces  the polynomial term in \eqref{lexp} with a  suitable rational function  expressed  in a convenient form as  sum of ratios.

The paper is organized as follows.  In Section \ref{acc} we derive the expansion formula \eqref{lexp} in the framework of Lanczos representation.  In Section \ref{err} we adjust known error estimates for  the Lanczos approximation to the function $q(\tau, w)$ by devising  a  suitable scheme for evaluating the pointwise convergence  in regions away from the discontinuities.   In Section \ref{accel} we show that this scheme  indeeed  yields an effective acceleration scheme involving only the manipulation of rational functions. In Section \ref{numexp} we illustrate the effectiveness of this scheme by numerical experiments related with the solution of a non-local BVP of type \eqref{l1}-\eqref{l2}.
Finally, conclusions and future work are drawn in Section \ref{end}.  

\section{Lanczos Acceleration of Fourier Expansions}\label{acc}
Suppose we set a value for $w,$ 
since $q(\tau, w)$ is analytic over the unit interval, it can be expanded into Fourier series of $\tau$ and we get
\[
q(\tau, w) = {\hat c}_0 + \sum_{k =1}^{+\infty} \left[{\hat c}_k \cos(2 \pi k \tau) + {\hat s}_k \sin(2 \pi k \tau) \right];
\]
the coefficients that occur in this series (Fourier coefficients) are given by
\[
{\hat c}_0 = \int_0^1 q_\tau(w)  d \tau =1
\]
and through the real and imaginary parts of
\[
{\hat c}_k + \mc {\hat s}_k =  \int_0^1 q_\tau(w) e^{\mc 2 \pi k \tau} d \tau.
\]
From the evaluation of the latter integral it therefore follows that
\begin{equation}\label{expan}
 q(\tau,w) =1+ 2w \sum_{k =1}^{+\infty}  \left[w \cos(2 \pi k \tau) -k  \sin(2 \pi k \tau) \right] (w^2+k^2)^{-1}.
\end{equation}
Since $q(\tau,w)$ is non-periodic, the use of the Fourier series has a computational disadvantage, i.e. the series can converge very slowly.

In \cite[Lemma 1]{Boito}, the authors developed methods to accelerate the convergence of \eqref{expan} by using relations expressing Bernoulli polynomials by appropriate numerical series of sine and cosine. This led to the relationship given in \eqref{lexp}.

In this section, a derivation of the so-called Lanczos representation for the function $q(\tau, w)$ is described. This involves expressing  $q(\tau, w)$ in the form
\begin{equation} \label{polf}
    q(\tau, w) \equiv g(\tau) = h_{p-1} (\tau) + f_p(\tau)
\end{equation}
where $h_{p-1}(\tau)$ is a polynomial of degree $(p-1)$ and $f_p(\tau)$ has Fourier coefficients with a decay of order $k^{-p}$ as $k \rightarrow +\infty.$
Following the results reported in \cite[pp. 82, 83]{Lyn}, for any $p\geq 1$  we get   
\begin{eqnarray}
h_{p-1} (\tau) &:=& 1+\sum_{k=1}^{p-1} [g^{(k-1)}(1)-g^{(k-1)}(0)] \frac{B_{k}(\tau)}{k!}, \label{lanczos}\\
 f_p(\tau)&:=& 2 \sum_{k=1}^{+\infty} \left[ c_k \cos(2 \pi k \tau) +  s_k \sin(2 \pi k \tau) \right], \label{lanczos1}
\end{eqnarray}
where $B_k(\cdot)$ is the k-th Bernoulli polynomial and the coefficients $c_k$ and $s_k$ 
are such that 
\[
c_k + \mc s_k=\frac{(-1)^p}{(2\pi k \mc)^p }  \int_0^1 g^{(p)}(\tau)(e^{2 \pi k \mc \tau}-1) d \tau,
\]
cfr. \cite[Eqs. (1.14) and (1.20)]{Lyn}. By direct calculation it is easily found that
\begin{eqnarray} \label{cksk}
    c_k + \mc  s_k &\,=&\frac{(-1)^p}{(2\pi k \mc)^p } \left (\frac{w^{p+1}}{e^w-1} \frac{e^{w+2\pi k \mc}-1}{w+2\pi k \mc} -w^p \right) \nonumber\\
&\,=&\frac{(-1)^p}{(2\pi k \mc)^p } w^{p+1} \left (\frac{1}{w+2\pi k \mc} -\frac{1}{w} \right) \nonumber\\
&\,=&\frac{(-1)^p}{(2\pi k  \mc)^p } w^{p+1} \left (\frac{w-2\pi k \mc}{w^2+(2\pi k)^2} -\frac{1}{w} \right) \nonumber\\
&\,=&\frac{(-1)^{p+1}}{(2\pi k)^{p-1} \mc^p } w^p \frac{ 2\pi k + w \mc }{ w^2+(2\pi k)^2 }.
\end{eqnarray}
Consequently,
\begin{eqnarray} \label{ckskparidispari}
  c_k + \mc  s_k &\,=& \left\{\begin{array}{cc}
   \displaystyle{\frac{(-1)^{p/2+1}}{(2\pi k)^{p-2} } \frac{ w^p  }{ w^2+(2\pi k)^2 }  + \mc \frac{(-1)^{p/2+1}}{(2\pi k)^{p-1} } \frac{ w^{p+1}  }{ w^2+(2\pi k)^2 } }, &  \, p \mbox{ even}, \\
   & \\
 \displaystyle{\frac{(-1)^{(p+1)/2}}{(2\pi k)^{p-1} } \frac{ w^{p+1}  }{ w^2+(2\pi k)^2 }  + \mc \frac{(-1)^{(p+1)/2}}{(2\pi k)^{p-2} } \frac{ w^p  }{ w^2+(2\pi k)^2 } }, &  \, p \mbox{ odd}.
 \end{array} \right.
\end{eqnarray}
In addition, observing that
\begin{equation*}\label{one}
g^{(k-1)}(1)-g^{(k-1)}(0) =  w^{k},  
\end{equation*}
and  that $B_0(\tau)=1,$ we can rewrite (see \eqref{lanczos})
\begin{equation*} \label{pez1}
\begin{split}
  h_{p-1} (\tau) &\; = B_0(\tau)+ \sum_{k=1}^{p-1} B_k(\tau) \frac{w^k}{k!} \\
&\; = \sum_{k=0}^{p-1} B_k(\tau) \frac{w^k}{k!}.  
\end{split}
\end{equation*}
To establish whether the newly introduced Lanczos representation of $q(\tau,w)$ recovers the expression found in \cite{Boito} and reported in \eqref{lexp}, we set $p= 2(n+1)$ and $w = (2\pi z).$ Obviously, 
\begin{eqnarray} \label{hn}
     h_{2n+1} (\tau) = \sum_{k=0}^{2n+1} B_k(\tau) \frac{(2\pi z)^k}{k!}.
\end{eqnarray} 
Concerning $f_p(\tau),$ from \eqref{ckskparidispari} we have that 
 \begin{equation} \label{fn}
     f_{2n+2}(\tau) = 2 \sum_{k=1}^{+\infty} \left[c_k \cos(2 \pi k \tau) +  s_k \sin(2 \pi k \tau) \right],
 \end{equation}
with
\begin{equation}\label{three}
 c_k = \frac{(-1)^n}{ k^{2n} } \frac{ z^{2n+2}}{ z^2+ k^2}, \qquad  s_k = \frac{(-1)^n}{k^{2n+1} } \frac{ z^{2n+3} }{z^2+k^2}.
\end{equation}
Then, it is immediate to verify that the Lanczos representation 
\[
q(\tau, w) = h_{2n+1} (\tau) + f_{2n+2} (\tau)
\]
agrees with \eqref{lexp}. 
 
Such concurrence is interesting from both a theoretical and a practical point of view. Indeed, it makes  possible to adapt general results for the Lanczos representation of regular functions to the specific expansion \eqref{lexp} of $q(\tau, 2\pi z)$. These results include error estimates and the design of alternative acceleration schemes. The corresponding adaptations for $q(\tau, w)$ are the subject of the next sections.

\section{Error Estimate and Convergence}\label{err}
The connection between the function $q(\tau,w)$ and the Lanczos representations discussed in the previous section is now exploited to derive the error estimate and to perform a convergence analysis.

Denoting by 
\[
f_{p,N}(\tau) := 2 \sum_{k=1}^{N} \left[ c_k \cos(2 \pi k \tau) +  s_k \sin(2 \pi k \tau) \right]
\]
the truncation to the first $N$ terms of  $f_p(\tau)$ (see \eqref{lanczos1}), we may approximate 
\begin{equation} \label{approxq}
    q(\tau, w) \approx h_{p-1} (\tau) + f_{p,N}(\tau).
\end{equation}
Then,  we are able to provide an estimate for the error
\begin{eqnarray} \label{errgen}
     R_{p,N}(\tau, w)&:=&q(\tau, w)- [h_{p-1} (\tau) + f_{p,N}(\tau)] \nonumber\\
     &=& 2 \sum_{k=N+1}^{+\infty} \left[ c_k \cos(2 \pi k \tau) +  s_k \sin(2 \pi k \tau) \right].
\end{eqnarray}
\begin{theorem}\label{first}
For a given integer $p \ge 1$ the following estimate holds
\begin{equation}\label{limit}
  \lim_{N \rightarrow +\infty} N^{p-1/2} \| R_{p,N}(\tau,w) \|_2 =    |w|^p \frac{1 }{(2\pi)^{p}} \frac{{\sqrt{2}}}{\sqrt{{2p-1}}}, 
\end{equation}
where $\parallel \cdot \parallel_2$ is the $L^2$-norm.
\end{theorem}
\begin{proof}
Following \cite[Theorem 2.1]{Barkhudaryan} and using \eqref{cksk} we can write 
\[
R_{p,N}(\tau,w) = \sum_{|k| >N} F_{-k} e^{-\mc 2 \pi k \tau},
\]
with
\[
F_{-k}  =    c_k + \mc  s_k.
\]
Using the Parseval's identity, we get
\begin{eqnarray*}
\| R_{p,N}(\tau,w) \|^2_2  &=& \sum_{|k| >N} |F_{-k}|^2  \\
 &=& \frac{1}{(2 \pi)^{2p}} |w|^{2p+2}\sum_{|k| >N} \frac{1}{k^{2p}} \Bigg \vert  \frac{w-2\pi k \mc }{w^2+(2\pi k)^2} - \frac{1}{w }   \Bigg \vert^2 \\
&=& \frac{1}{(2 \pi)^{2p}} |w|^{2p+2}\sum_{|k| >N} \frac{1}{k^{2p}} \frac{(2\pi k)^2}{w^2 [w^2+(2\pi k)^2]} \\
&\approx& \frac{1}{(2\pi)^{2p}}  |w|^{2p}\sum_{|k| >N} \frac{1}{k^{2p}} \\
&\approx& \frac{1}{(2\pi)^{2p}} |w|^{2p}  \frac{ {2} N^{-2p+1}}{2p-1}. 
\end{eqnarray*}
This concludes the proof.  $\qedwhite$
\end{proof} 
\begin{remark}
    It is worth noting in \eqref{limit} the occurrence   of the term $|w|^p$ which  can be problematic for convergence in the $w-$domain.  This observation is confirmed by many experimental  results reported in \cite{GL} which indicate that the more is the size of $w$, the more is the number of terms of \eqref{lexp} needed to reach a prescribed accuracy.  
\end{remark}
When $p=2(n+1)$  and $w=(2 \pi z),$ \eqref{approxq} coincides with  \eqref{lexpN}. So, we can rewrite the approximation to $q(\tau,2\pi z)$ given in \eqref{lexpN} as
\begin{equation} \label{Bern}
  g_{n,N}(\tau,z) =  h_{2n+1} (\tau)+ f_{2n+2,N}(\tau).
\end{equation}
Using these notation,  we define the error as
\begin{align}\label{error}
& R_{n,N}(\tau, z):=q(\tau, 2 \pi z)-g_{n, N}(\tau, z) \nonumber\\
&=2 \sum_{k=N+1}^{+\infty} \left[\frac{ (-1)^n  z^{2n+2}}{ k^{2n} (z^2+ k^2)}  \cos(2 \pi k \tau)+ \frac{  (-1)^n  z^{2n+3} }{k^{2n+1} (z^2+k^2)} \sin(2 \pi k \tau)\right].
\end{align}
From Theorem \ref{first} it is immediate to deduce that 
\[
\lim_{N \rightarrow +\infty} N^{2n+3/2} \| R_{n,N}(\tau,z) \|_2 = |z|^{2n+2}    \frac{{\sqrt{2}}}{\sqrt{4n+3}}.
\]
In Table \ref{t1} we show a numerical illustration of this result with $n=1$. For a given $z\in \{1, 0.1, 10\}$ and $N\in \{512, 1024, 2048\}$ we compute the quantity
\begin{equation} \label{Dn}
 \Delta(N)= \frac{N^{7/2} }{|z|^4} \parallel R_{1,N}(\tau, z) \parallel_2,   
\end{equation}
where $\parallel R_{1,N}(\tau, z) \parallel_2$ is estimated using the \textsc{Mathematica} function {\tt NIntegrate} applied to the partial sum formed by the first $2048$ terms of the expansion \eqref{error}. Doubling the number of terms does not change the results.  The integrands  are  highly oscillatory and \textsc{Mathematica} returns partial errors of order $10^{-3}$. 
As can be seen, for each $N$ considered the computed  value $\Delta(N)$ is a good approximation of $\sqrt{2/7}\approx 0.5345,$ independently of $z.$ \\
\begin{table}
\centering
  \begin{tabular}{|c||c|c|c|}
\hhline{-||---}
{\backslashbox{$z$}{$N$}}
  &  512 & 1024& 2048\\
\hhline{=::===}
\multirow{1}{*}{1}
 &$0.5327$ & $0.5328$& $0.5319$\\
\hhline{-||---}
\multirow{1}{*}{0.1}
 &$0.5327$ & $0.5328$ & $0.5319$\\
\hhline{-||---}
\multirow{1}{*}{10}
 &$0.5327$ & $0.5328$ & $0.5319$\\
\hhline{-||---}
\end{tabular}
\caption{Values of the quantity $\Delta(N)$ given in \eqref{Dn}  with $R_{1,N}(\tau, z)$ truncated at the first $2048$ terms.}
\label{t1}
\end{table}

To investigate pointwise convergence in regions away from the discontinuities at the endpoints of the interval $[0, 1]$, we follow an approach that is based on the results of \cite{KW} and differs from the approach used in \cite{Boito}. This approach avoids the use of Bernoulli polynomials by employing rational  corrections of the error in order to accelerate its  convergence toward zero. 
More precisely, we consider the first term in \eqref{errgen} without considering a constant, namely (see \eqref{ckskparidispari})
\begin{eqnarray} \label{error1}
 R_{p,N}^{(1)}(\tau, w)&:=&
 \left\{\begin{array}{cc}
 \displaystyle \sum_{k=N+1}^{+\infty} 
\frac{1}{(2\pi k)^{p-2} } \frac{ w^p  }{ w^2+(2\pi k)^2 } \cos(2 \pi k \tau),  &  \quad p \mbox{ even}, \nonumber\\
   & \\
 \displaystyle \sum_{k=N+1}^{+\infty} {\frac{1}{(2\pi k)^{p-1} } \frac{ w^{p+1}  }{ w^2+(2\pi k)^2 }
\cos(2 \pi k \tau)}, &  \quad p \mbox{ odd},
 \end{array} \right.\\
&\equiv& \sum_{k=N+1}^{+\infty} \gamma_{k}^{(0)}{ \cos(2 \pi k \tau)}.
\end{eqnarray}
From the relation 
\[
2 \cos(2 \pi k \tau) \cos(2 \pi \tau) =\cos(2 \pi (k+1) \tau)+\cos(2 \pi (k-1) \tau), \qquad k\geq 1, 
\]
it follows that 
\begin{eqnarray*}
(2- 2\cos (2 \pi \tau) )R_{p,N}^{(1)}(\tau, w) &=& \displaystyle\sum_{k=N+1}^{+\infty} 2 \gamma_{k}^{(0)}  \cos(2 \pi k \tau) \\
&& - \displaystyle\sum_{k=N+1}^{+\infty} \gamma_{k}^{(0)} \left[\cos(2 \pi (k+1) \tau)+\cos(2 \pi (k-1) \tau)\right] \\
&=&\displaystyle\sum_{k=N+1}^{+\infty} 2 \gamma_{k}^{(0)}  \cos(2 \pi k \tau) \\
&& - \displaystyle\sum_{r=N+2}^{+\infty} \gamma_{r-1}^{(0)} \cos(2 \pi r \tau) - \displaystyle\sum_{s=N}^{+\infty} \gamma_{s+1}^{(0)} \cos(2 \pi s \tau).
\end{eqnarray*}
Therefore, we obtains that 
\begin{eqnarray} \label{relA}
    (2&-& 2\cos(2 \pi \tau) )R_{p,N}^{(1)}(\tau, w) =  \gamma_{N+1}^{(0)}\left[ 2 \cos(2 \pi (N+1) \tau) - \cos(2 \pi N \tau)\right] \nonumber\\
&& -\gamma_{N+2}^{(0)}  \cos(2 \pi (N+1) \tau) +\displaystyle\sum_{k=N+2}^{+\infty} \gamma_{k}^{(1)}\cos(2 \pi k \tau)  
\end{eqnarray}
with  
\begin{equation} \label{gammak1}
  \gamma_{k}^{(1)}=-\gamma_{k-1}^{(0)}+ 2 \gamma_{k}^{(0)}  -\gamma_{k+1}^{(0)}.  
\end{equation}

Notice that  $-\gamma_{k}^{(1)}$ is the classical approximation  of the second derivative of the function (see \eqref{error1} and \eqref{ckskparidispari})
\begin{eqnarray}\label{f1def}  
f_1(x) &\,=& \left\{\begin{array}{cc}
   \displaystyle{\frac{1}{(2\pi x)^{p-2} } \frac{ w^p  }{ w^2+(2\pi x)^2 }}, &  \quad p \mbox{ even}, \\
   & \\
 \displaystyle{\frac{1}{(2\pi x)^{p-1} } \frac{ w^{p+1}  }{ w^2+(2\pi x)^2 }}, &  \quad p \mbox{ odd},
 \end{array} \right.
\end{eqnarray}
over the grid $x_k=k$, $k\geq N+2$. Since 
\begin{eqnarray}  
f_1''(x) &\,=& \left\{\begin{array}{cc}
   \frac{\theta_1(p)w^px^{4-p}}{(w^2+(2\pi x)^2)^3} +  \frac{\theta_2(p)w^px^{2-p}}{(w^2+(2\pi x)^2)^2} +\frac{\theta_3(p)w^px^{-p}}{w^2+(2\pi x)^2} , &  \quad p \mbox{ even}, \\
   & \\
 \frac{\vartheta_1(p)w^{1+p}x^{3-p}}{(w^2+(2\pi x)^2)^3} +  \frac{\vartheta_2(p)w^{1+p}x^{1-p}}{(w^2+(2\pi x)^2)^2} +\frac{\vartheta_3(p)w^{1+p}x^{-1-p}}{w^2+(2\pi x)^2}, &  \quad p \mbox{ odd},
 \end{array} \right.
\end{eqnarray}
from Taylor expansion  we  deduce that 
\[
f_1''(k) + \gamma_{k}^{(1)} = -\frac{2 f_1^\mathrm{\romannumeral 4}(\xi_k)}{4!}, \qquad k-1\leq \xi_k\leq k+1, 
\]
which implies 
\begin{eqnarray}  \label{gammaak1stima}
\gamma_{k}^{(1)} =
\left\{\begin{array}{cc}
   {\mathcal{O}}\left(\displaystyle {k^{-p-2}}\right), &  \quad p \mbox{ even}, \\
   & \\
\mathcal{O}\left(\displaystyle k^{-p-3}\right), &  \quad p \mbox{ odd}.
 \end{array} \right. 
 \end{eqnarray}
Analogously,  for the second term in \eqref{errgen} we have
\begin{eqnarray} \label{error2}
R_{p,N}^{(2)}(\tau, w)&:=&  \left\{\begin{array}{cc}
 \displaystyle \sum_{k=N+1}^{+\infty} \frac{1}{(2\pi k)^{p-1}} \frac{w^{p+1} }{w^2+(2\pi k)^2} 
\sin(2 \pi k \tau),  &  \quad p \mbox{ even}, \nonumber\\
   & \\
 \displaystyle \sum_{k=N+1}^{+\infty} 
\frac{1}{(2\pi k)^{p-2}} \frac{w^p}{w^2+(2\pi k)^2}
\sin(2 \pi k \tau), &  \quad p \mbox{ odd},
 \end{array} \right.\\
&\equiv& \sum_{k=N+1}^{+\infty} \delta_{k}^{(0)}{ \sin(2 \pi k \tau)}.
\end{eqnarray}
Using the the formula 
\[
2 \sin(2 \pi k \tau) \cos (2 \pi\tau) =\sin(2 \pi (k+1) \tau)+\sin(2 \pi (k-1) \tau), \qquad k\geq 1, 
\]
we obtain that  
\begin{eqnarray}\label{relB}
    (2&-& 2\cos (2 \pi \tau) )R_{p,N}^{(2)}(\tau, w) =  \delta_{N+1}^{(0)}\left[ 2 \sin(2 \pi (N+1) \tau) - \sin(2 \pi N \tau)\right] \nonumber\\
    &-&\delta_{N+2}^{(0)}  \sin(2 \pi (N+1) \tau) +\displaystyle\sum_{k=N+2}^{+\infty} \delta_{k}^{(1)}\sin(2 \pi k \tau), 
\end{eqnarray}
with 
\begin{equation} \label{deltak1}
    \delta_{k}^{(1)}=-\delta_{k-1}^{(0)}+ 2 \delta_{k}^{(0)}  - \delta_{k+1}^{(0)}.
\end{equation}
Again,  we observe that  $-\delta_{k}^{(1)}$ is the classical approximation  of the second derivative of the function 
\begin{eqnarray}\label{f2def}  
f_2(x) &\,=& \left\{\begin{array}{cc}
  \displaystyle{\frac{1}{(2\pi x)^{p-1} } \frac{ w^{p+1}  }{ w^2+(2\pi x)^2 }},  &  \quad p \mbox{ even}, \\
   & \\
\displaystyle{\frac{1}{(2\pi x)^{p-2} } \frac{ w^p  }{ w^2+(2\pi x)^2 }},  &  \quad p \mbox{ odd},
 \end{array} \right.
\end{eqnarray}
over the grid $x_k=k$, $k\geq N+2$. 
Notice that \eqref{f1def} coincides with \eqref{f2def} up to a change of parity of $p$. Hence, from \eqref{gammaak1stima}
we find that 
\begin{eqnarray}  \label{deltak1stima}
\delta_{k}^{(1)} =
\left\{\begin{array}{cc}
   \mathcal O\left(\displaystyle {k^{-p-3}}\right), &  \quad p \mbox{ even}, \\
   & \\
\mathcal O\left(\displaystyle k^{-p-2}\right), &  \quad p \mbox{ odd}.
 \end{array} \right.
\end{eqnarray}
In this way, we arrive at the following result which describes the behaviour  of the error in the interior of the unit interval. 
\begin{theorem}\label{second} 
For any $\tau\in (0, 1),$ setting
\[
m = \left\{\begin{array}{ll}
    p, & \,\, p \mbox{ even}, \\
    p+1, & \, \,p \mbox{ odd},
\end{array}\right.
\]
we have 
    \[
    R_{p,N}(\tau, w)= \frac{(-1)^{\left \lceil{(p+1)/2}\right \rceil} }{ 1- \cos (2\pi \tau)} \Gamma_{p,N}(\tau,w) + \smallO \left(N^{-m}\right), \quad  N\rightarrow +\infty,
    \]
    where 
$\left \lceil{\cdot}\right \rceil$ denotes the ceiling function and
    \[
    \Gamma_{p,N}(\tau,w)=
    \gamma_{N+1}^{(0)}\left[ 2 \cos(2 \pi (N+1) \tau) - \cos(2 \pi N \tau)\right]  -\gamma_{N+2}^{(0)}  \cos(2 \pi (N+1) \tau).
    \]
\end{theorem}
\begin{proof}
Using \eqref{error1} and \eqref{error2}, the error in \eqref{errgen} can be written as
\begin{equation} \label{twoerr}
    R_{p,N}(\tau, w) = (-1)^{\left \lceil{(p+1)/2}\right \rceil} 2 [R^{(1)}_{p,N}(\tau, w)+R^{(2)}_{p,N}(\tau, w)].
\end{equation}
Then, taking into account the relations given in \eqref{relA} and \eqref{relB},  and the estimates \eqref{gammaak1stima}, \eqref{deltak1stima} the result follows immediately.  $\qedwhite$
\end{proof} 

\section{Acceleration of Convergence}\label{accel}
The results reported in Theorem \ref{first} and  Theorem \ref{second}   illustrate  theoretically the acceleration  capabilities  of Lanczos approximation.  
However, the numerical results shown in  \cite{Boito3}  clearly indicate that the computation of the polynomial term in \eqref{lexp}  can be prone to numerical instabilities.  To circumvent this issue, in this section we investigate the design of  different acceleration techniques for evaluating $q(\tau, w)$.  Our starting point is the observation that the strategy devised  in the previous section for  the error analysis   of the Lanczos representation can be  applied from scratch
to the expansion with $p=2$ or very small values of $p$. \\ 
Recalling that (see \eqref{error1} and \eqref{error2})
\[
\gamma_k^{(0)}=
\left\{\begin{array}{cc}
 \displaystyle
\frac{1}{(2\pi k)^{p-2}} \frac{ w^p  }{ w^2+(2\pi k)^2 },  &  \, p \mbox{ even}, \\
   & \\
 \displaystyle  {\frac{1}{(2\pi k)^{p-1} } \frac{ w^{p+1}  }{ w^2+(2\pi k)^2 }}, &  \, p \mbox{ odd},
 \end{array} \right.
 \]
 \[
 \delta_k^{(0)}= 
 \left\{\begin{array}{cc}
 \displaystyle
\frac{1}{(2\pi k)^{p-1}} \frac{w^{p+1} }{w^2+(2\pi k)^2},  &  \, p \mbox{ even}, \\
   & \\
\displaystyle \frac{1}{(2\pi k)^{p-2}} \frac{w^p}{w^2+(2\pi k)^2}, &  \, p \mbox{ odd},
 \end{array} \right.
\]
and denoting by 
\[
\Gamma_1^{(1)}:= \frac{\gamma_{N+1}^{(0)}\left[ 2 \cos(2 \pi (N+1) \tau) - \cos(2 \pi (N) \tau)\right] -\gamma_{N+2}^{(0)} \cos((N+1)\tau)}{2-2 \cos(2\pi \tau)}
\]
and 
\[
\Delta_1^{(1)}:=  \frac{\delta_{N+1}^{(0)}\left[ 2 \sin(2 \pi (N+1) \tau) - \sin(2 \pi (N) \tau)\right]  -\delta_{N+2}^{(0)} \sin((N+1)\tau)}{2-2 \cos(2\pi \tau)},
\]
from  \eqref{relA} and \eqref{relB} we find that  for any $\tau\in(0, 1)$ the error can be expressed as (see \eqref{twoerr})
\begin{eqnarray*}\label{res1}
R_{p,N}(\tau, w) &=&   \frac{(-1)^{\left \lceil{(p+1)/2}\right \rceil} 2 }{2-2 \cos(2\pi \tau)} 
 [(2-2 \cos(2\pi \tau)) R^{(1)}_{p,N}(\tau, w)+\\
 &+&(2-2 \cos(2\pi \tau))R^{(2)}_{p,N}(\tau, w)]\\
&=&    (-1)^{\left \lceil{(p+1)/2}\right \rceil}  2 \left(  \Gamma_1^{(1)} + \Delta_1^{(1)} \right)+\\
&+& \frac{(-1)^{\left \lceil{(p+1)/2}\right \rceil}  2 }{2- 2\cos(2\pi \tau)} \left(
\displaystyle\sum_{k=N+2}^{+\infty}\gamma_{k}^{(1)}\cos(2\pi k \tau) +\displaystyle\sum_{k=N+2}^{+\infty}\delta_{k}^{(1)}\sin(2\pi k \tau)\right),
\end{eqnarray*}
with
\[
\gamma_{k}^{(j)}=-\gamma_{k-1}^{(j-1)}+ 2 \gamma_{k}^{(j-1)} - \gamma_{k+1}^{(j-1)}, \qquad \quad
\delta_{k}^{(j)}=-\delta_{k-1}^{(j-1)}+ 2 \delta_{k}^{(j-1)} -  \delta_{k+1}^{(j-1)}, \qquad j \ge 1.
\]
The process can be iterated in order to accelerate the convergence of  the two series  in the left hand side  of this relation by the multiplication by the factor $2- 2\cos(2\pi \tau)$.
In this way, 
by setting for an integer $\ell\ge 1$,  
\[
\Gamma_1^{(\ell)}= \sum_{j=1}^\ell \frac{\gamma_{N+j}^{(j-1)}\left[ 2 \cos(2 \pi (N+j) \tau) - \cos(2 \pi (N+j-1) \tau)\right] -\gamma_{N+j+1}^{(j-1)} \cos((N+j)\tau)}{(2-2 \cos (2\pi \tau))^j}
\]
and 
\[
\Delta_1^{(\ell)}= \sum_{j=1}^\ell \frac{\delta_{N+j}^{(j-1)}\left[ 2 \sin(2 \pi (N+j) \tau) - \sin(2 \pi (N+j-1) \tau)\right]  -\delta_{N+j+1}^{(j-1)} \sin((N+j)\tau)}{(2-2 \cos (2\pi \tau))^j},
\]
we find that
\[
R_{p,N}(t, w)= (-1)^{\left \lceil{(p+1)/2}\right \rceil} 2  \left( \Gamma_1^{(\ell)} + \Delta_1^{(\ell)} \right) +S_{p,N, \ell}(\tau,w),\qquad \ell \ge 1,
\]
or, equivalently,
\begin{eqnarray}
     q(\tau, w) &=& [h_{p-1} (\tau) + f_{p,N}(\tau)]  + (-1)^{\left \lceil{(p+1)/2}\right \rceil} 2  \left( \Gamma_1^{(\ell)} + \Delta_1^{(\ell)} \right) + S_{p,N, \ell}(\tau, w) \nonumber\\
     &:=& \mathcal{G}_{p, N,\ell}(\tau, w) + S_{p,N, \ell}(\tau, w), \qquad \ell \ge 1. \label{accapp}
\end{eqnarray}
\begin{remark}
At $\ell=0$ the approximation of $q(\tau,w)$ reduces to the approximation already given in \eqref{errgen} without acceleration, i.e.
 \[
 \mathcal{G}_{p, N,0}(\tau, w)= h_{p-1} (\tau) + f_{p,N}(\tau).
 \]
\end{remark}
\begin{remark}\label{comput}
The  approximation in  \eqref{accapp}  of  $q(\tau, w)$ requires the incorporation of the corrective term $(-1)^{\left \lceil{(p+1)/2}\right \rceil} 2  \left( \Gamma_1^{(\ell)} + \Delta_1^{(\ell)} \right)$. The evaluation of this term boils down to the calculation of quantities
 \begin{equation} \label{quantitagamma}
    \gamma_{N+1}^{(0)}, \gamma_{N+2}^{(0)}, \gamma_{N+2}^{(1)}, \gamma_{N+3}^{(1)}, \ldots, \gamma_{N+\ell}^{(\ell-1)}, \gamma_{N+\ell+1}^{(\ell-1)},  
 \end{equation}
 and 
  \begin{equation*} \label{quantitadelta}
 \delta_{N+1}^{(0)}, \delta_{N+2}^{(0)}, \delta_{N+2}^{(1)}, \delta_{N+3}^{(1)}, \ldots, \delta_{N+\ell}^{(\ell-1)}, \delta_{N+\ell+1}^{(\ell-1)}.
  \end{equation*}
 All these quantities can be obtained from 
 $\gamma_{N+j}^{(0)}$ and $\delta_{N+j}^{(0)}$, $0\leq j\leq 2\ell$, through recurrence relations. Furthermore, $\delta_{N+j}^{(0)}$ is derived from $\gamma_{N+j}^{(0)}$ at the cost of additional matrix-by-vector multiplication. The next graph illustrates the computation of the required quantities in \eqref{quantitagamma} when  $\ell=3$. 
\begin{center}
\vspace{4mm}
\begin{tikzpicture}[node distance= 10mm and 8mm, thick, main/.style = {draw, circle}] 
\node[main] (1) {$\gamma_{N+1}^{(0)}$}; 
\node[main] (2) [right=of 1] {$\gamma_{N+2}^{(0)}$}; 
\node[main] (3) [right=of 2] {$\gamma_{N+3}^{(0)}$};
\node[main] (4) [right=of 3] {$\gamma_{N+4}^{(0)}$};
\node[main] (5) [right=of 4] {$\gamma_{N+5}^{(0)}$};
\node[main] (6) [right=of 5] {$\gamma_{N+6}^{(0)}$};
\node[main] (7) [above=of 2] {$\gamma_{N+2}^{(1)}$}; 
\node[main] (8) [above=of 3] {$\gamma_{N+3}^{(1)}$}; 
\node[main] (9) [above=of 4] {$\gamma_{N+4}^{(1)}$}; 
\node[main] (10) [above=of 5] {$\gamma_{N+5}^{(1)}$};
\node[main] (11) [above=of 8] {$\gamma_{N+3}^{(2)}$};
\node[main] (12) [above=of 9] {$\gamma_{N+4}^{(2)}$};
\draw[->] (1) -- (7); 
\draw[->] (2) -- (7); 
\draw[->] (3) -- (7); 
\draw[->] (2) -- (8); 
\draw[->] (3) -- (8); 
\draw[->] (4) -- (8); 
\draw[->] (3) -- (9); 
\draw[->] (4) -- (9); 
\draw[->] (5) -- (9); 
\draw[->] (4) -- (10); 
\draw[->] (5) -- (10); 
\draw[->] (6) -- (10); 
\draw[->] (7) -- (11); 
\draw[->] (8) -- (11); 
\draw[->] (9) -- (11); 
\draw[->] (8) -- (12); 
\draw[->] (9) -- (12); 
\draw[->] (10) -- (12);
\end{tikzpicture}
\vspace{2mm}
\end{center}
Therefore, considering that the evaluation of 
$\mathcal{G}_{p,N,0}(\tau,A)\B f$ essentially requires the solution of $N$ shifted  linear systems of the form $(A^2 + k^2 I) \B x = \B b$ while the computation  of $\mathcal{G}_{p,N,\ell}(\tau,A)\B f$ requires solving $2\ell $ additional systems, we conclude that for small values of $\ell$  both approximations are comparable in terms of computational cost.  Furthermore, it is worthy noticing that the  quantities \eqref{quantitagamma} can be computed once and then reused for evaluating $q(\tau, A)$ for any $\tau$.
\end{remark}

The error estimates for the residual $S_{p,N, \ell}(\tau, w)$ can be derived using the same arguments exploited in the proofs of  Theorem \ref{first}  and Theorem \ref{second}. A global estimate of $\parallel S_{p,N, \ell}(\tau, w) \parallel_2$ can be obtained by introducing the penalty factors $\theta_j$, $0<\theta_j< 1$, as pursued in \cite{Pogh,Pogh1}. In particular, the error analysis is extended by replacing the term $2-2 \cos (2\pi \tau)$ with the modified form $2-2 \theta_j \cos (2\pi \tau)$ for appropriate choices of such penalty factors. The overall process behaves similarly with the only modification being the use of weighted  finite difference approximations for the second derivatives.  

We have performed numerical experiments to test the effectiveness of the proposed acceleration technique. In Figure \ref{f1}  we illustrate this behavior by drawing the relative error function
\[
err_{p,N, \ell}(\tau,w)=\frac{|q(\tau,w)-\mathcal{G}_{p,N, \ell}(\tau, w)|} {|q(\tau,w)|}, \quad -10\leq w\leq 0,
\]
for $\tau \in \{2^{-3}, 2^{-7}\},$ by setting $N=100,$ and for different values of $p$ and $\ell$. 
Our experiments  indicate that the method performs quite well for $\tau$ away from the discontinuities  and its effectiveness decreases as  the value of $\tau$ approaches $0$.   
\begin{figure}
  \centering
\subfloat[$\ell=0$]{\includegraphics[width=0.5\textwidth]{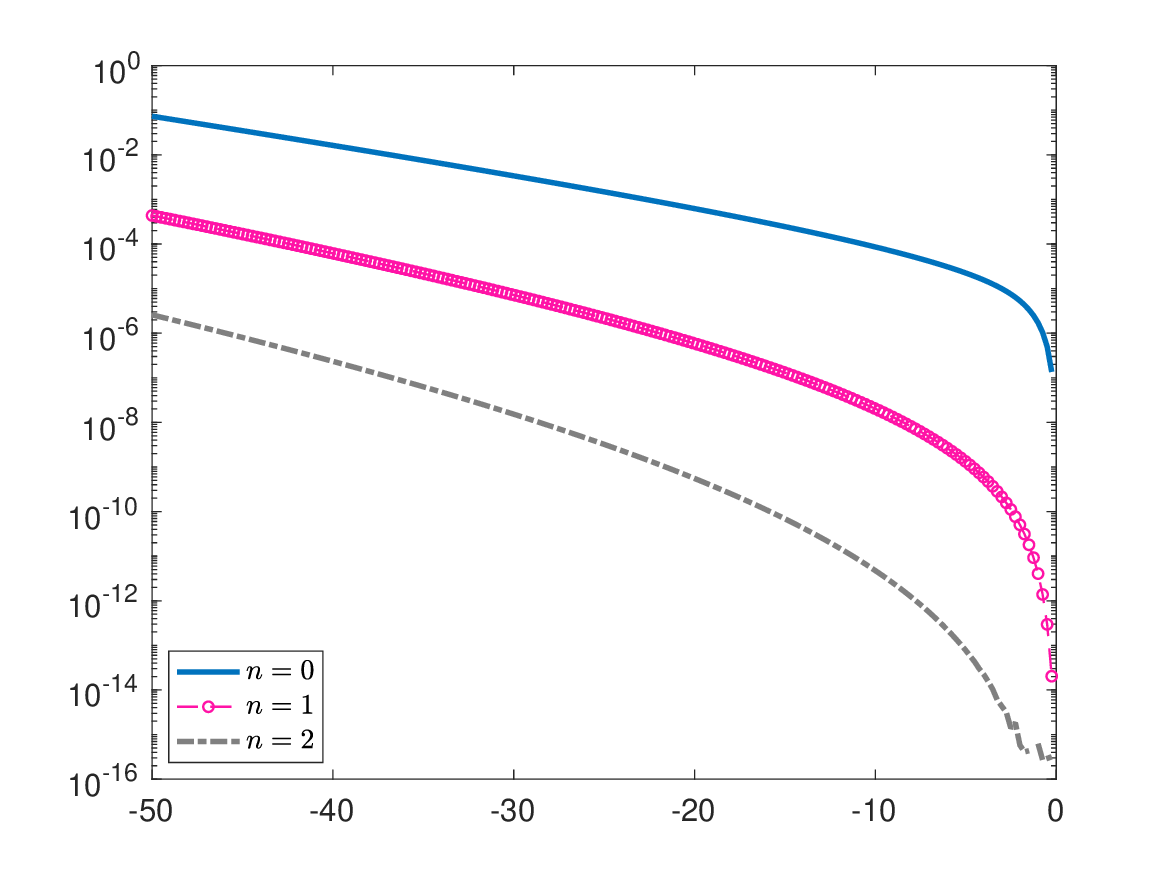}}
\hfill
\subfloat[$p=2$]{\includegraphics[width=0.5\textwidth]{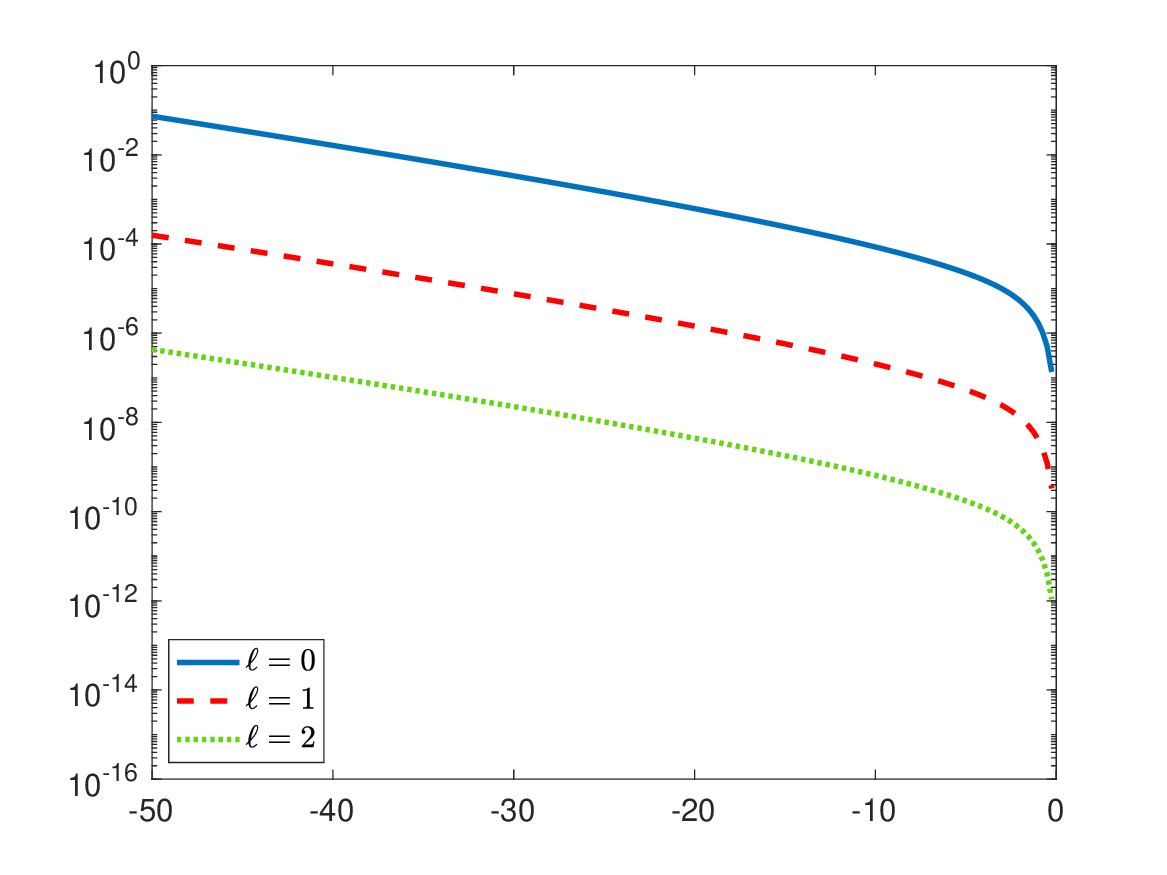}}
\hfill
\subfloat[$\ell=0$]{\includegraphics[width=0.5\textwidth]{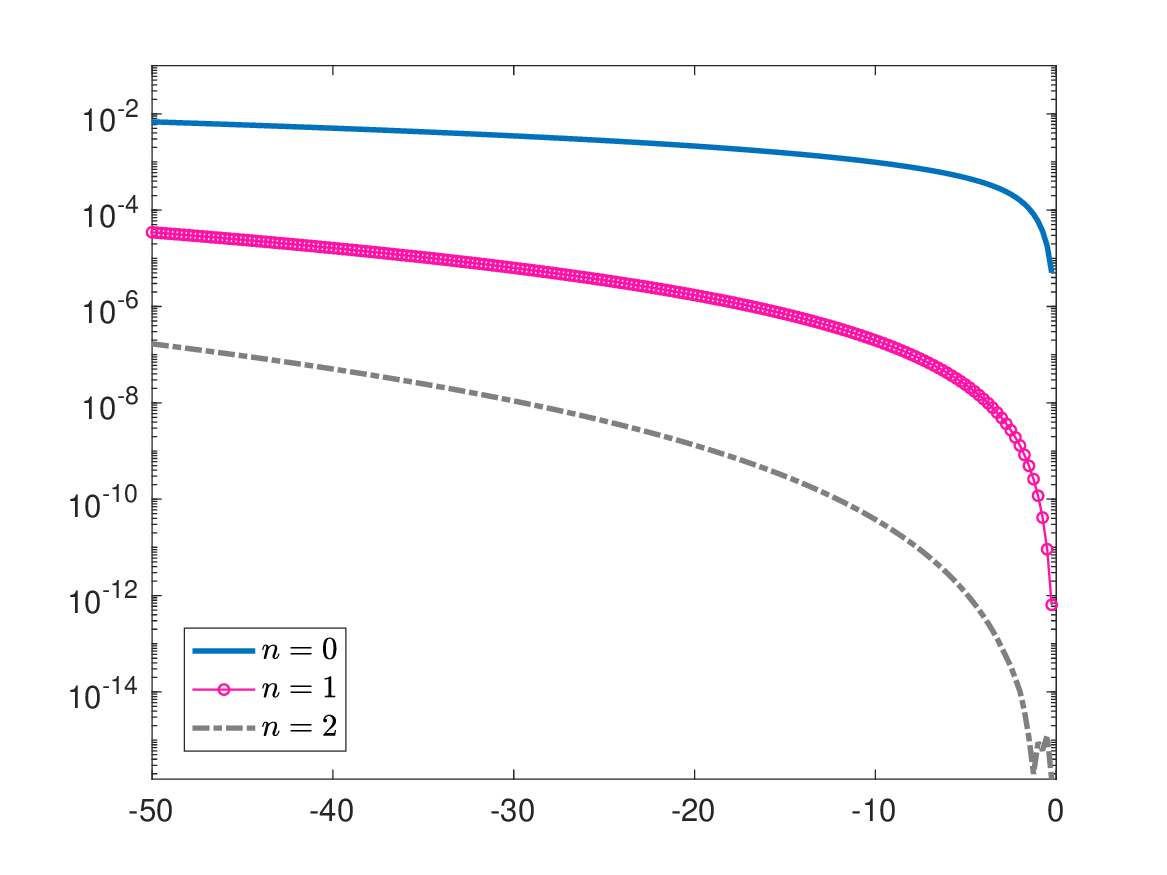}}
\hfill
\subfloat[$p=2$]{\includegraphics[width=0.5\textwidth]{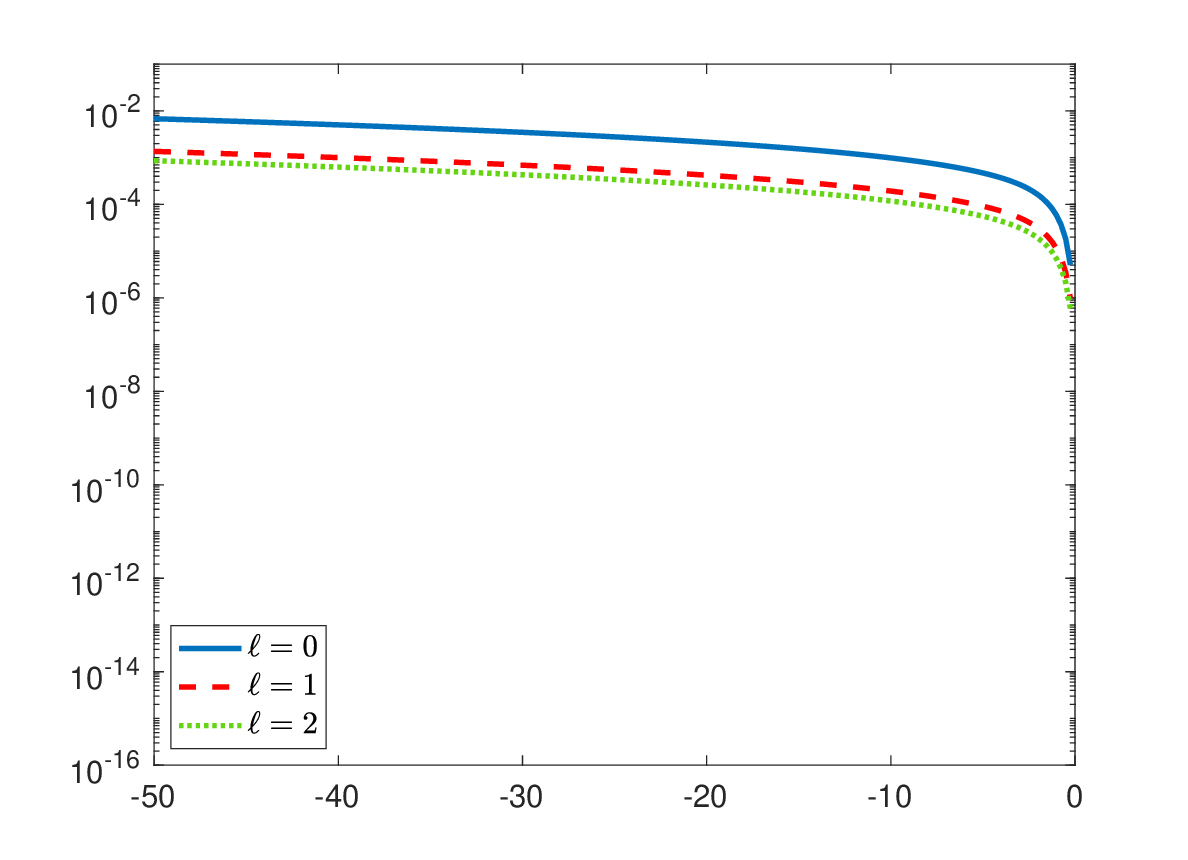}}
  \caption{Plots of the relative error $err_{p,N, \ell}(\tau,w)$  with $N=100$ and $p=2n+2.$ \\In (a)-(b): $\tau = 2^{-3}.$ In (c)-(d): $\tau = 2^{-7}.$}
\label{f1}
\end{figure}
If $\tau=0,$ to get reliable algorithm for approximating
$q(0,w)$ we must combine our approach 
with a suitable numerical method for evaluating the exponential function. Indeed, we can rely on the identity
\begin{equation}\label{formula}
q(0,w) 
= q(1-\alpha,w) e^{w/\alpha}-w, \qquad \alpha \in (0,1).
\end{equation}
The resulting approach is considered for $\alpha=1/8$ in Figure \ref{f2} where the exponential function is computed using  the \textsc{Matlab} built-in  function \texttt{exp}.
\begin{figure}
  \subfloat[$\ell = 0$]{\includegraphics[width=0.5\textwidth]{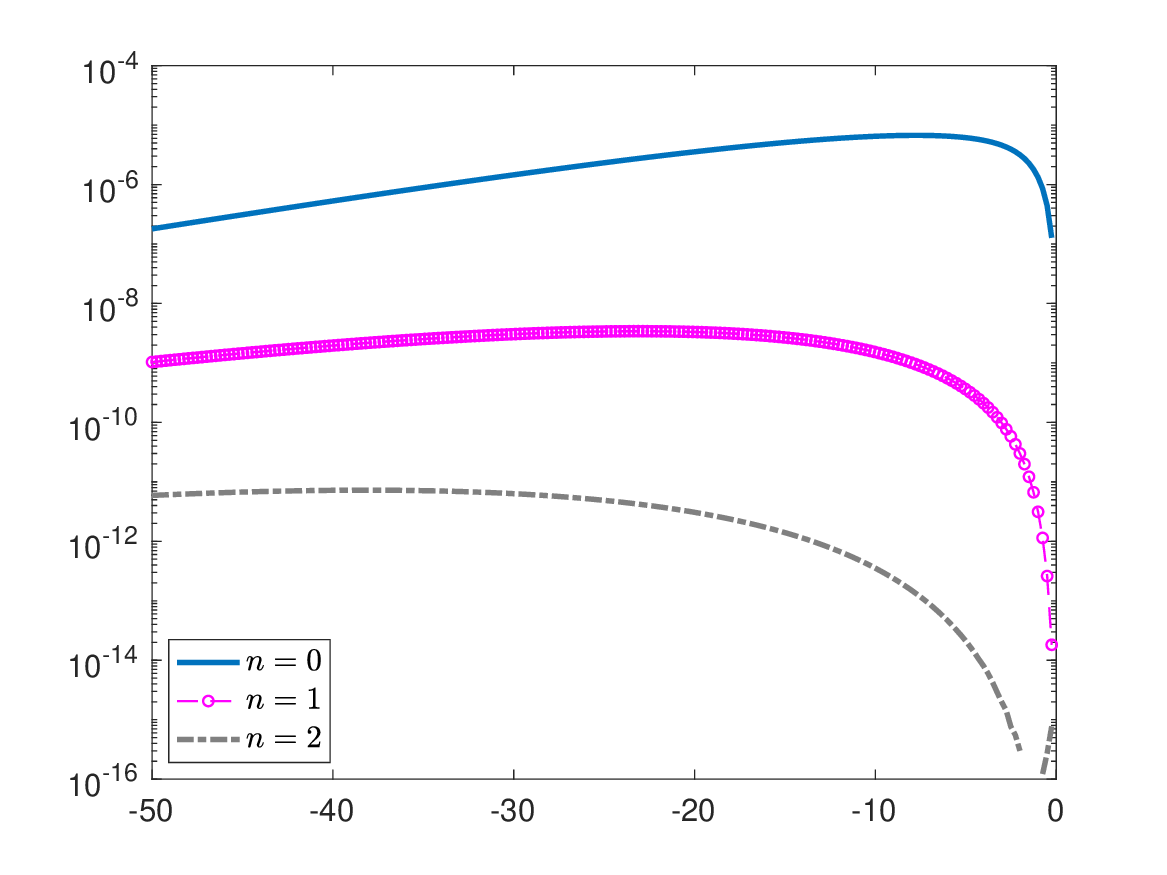}}
   \hfill
   \subfloat[$p = 2$]{\includegraphics[width=0.5\textwidth]{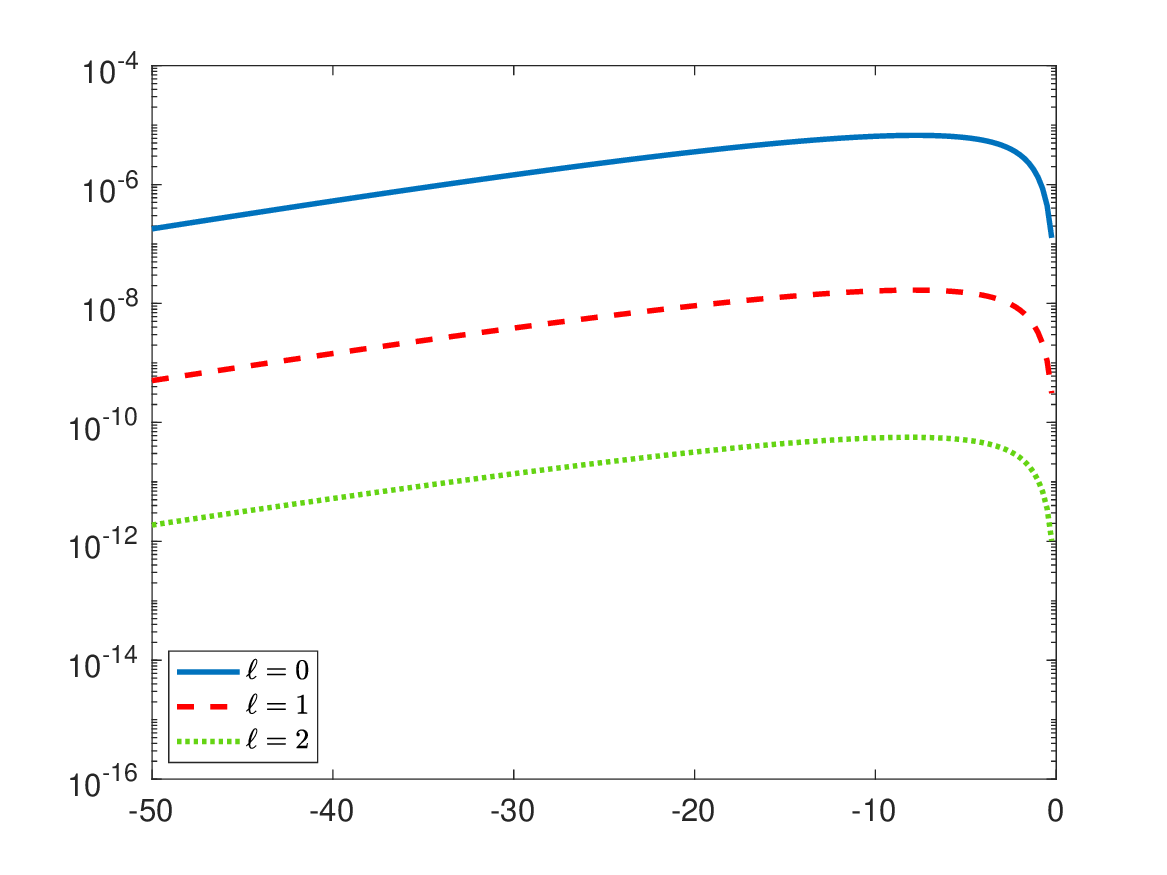}}
   \caption{Plots of $err_{p,N,\ell}(\tau,w)$  at $\tau=0$ generated  by  using formula \eqref{formula} with $\alpha=1/8,$ $N=100$ and $p=2n+2.$}
\label{f2}
\end{figure}
Another popular method for calculating the exponential is the rational approximation. This tool can be suitable for parallel computations \cite{GALSA}. For further experiments, we used the rational Chebyshev approximation
$\mathcal{N}_r(\zeta)/\mathcal{D}_r(\zeta)$ from $e^{-\zeta}$ for $\zeta\in [0, +\infty)$, with $\mathcal{N}_r(\zeta)$ and $\mathcal{D}_r(\zeta)$ polynomials of degree $r$, where the coefficients of such polynomials up to degree $r=30$ are given in \cite{CRV}. 
In this way, we obtain an approximation of $q(0,w)$, denoted as $u_{n,N, \ell, r,\alpha}(w)$ and defined by
\[
 q(0,w) \approx u_{p,N, \ell, r,\alpha}(w):=
 \mathcal{G}_{p,N, \ell}(\alpha, w) \frac{\mathcal{N}_r(-w/\alpha)}{\mathcal{D}_r(-w/\alpha)}-w.
 \]
Considering that the rational Chebyshev approximation of $e^{-\zeta}$ has uniform absolute error estimates of the form
\[
\sup_{\zeta \geq 0} \Bigg|\frac{\mathcal{N}_r(\zeta)}{\mathcal{D}_r(\zeta)}-e^{-\zeta} \Bigg|\simeq 10^{-r},
\]
cfr. \cite{CRV}, 
in Figure \ref{f3} we plot the absolute error function 
\[
err_{p,N, \ell, r,\alpha}(w)=|q(0,w)-u_{p,N, \ell, r,\alpha}(w)|, \qquad -5000 \leq w\leq 0 ,
\]
with $N=100,$  $r= 13, \alpha= 1/8$ and for different values of the parameters $p$ and $\ell.$  
\begin{figure}
  \centering
  \subfloat[$(p,\ell)=(6,0)$]{\includegraphics[width=0.5\textwidth]{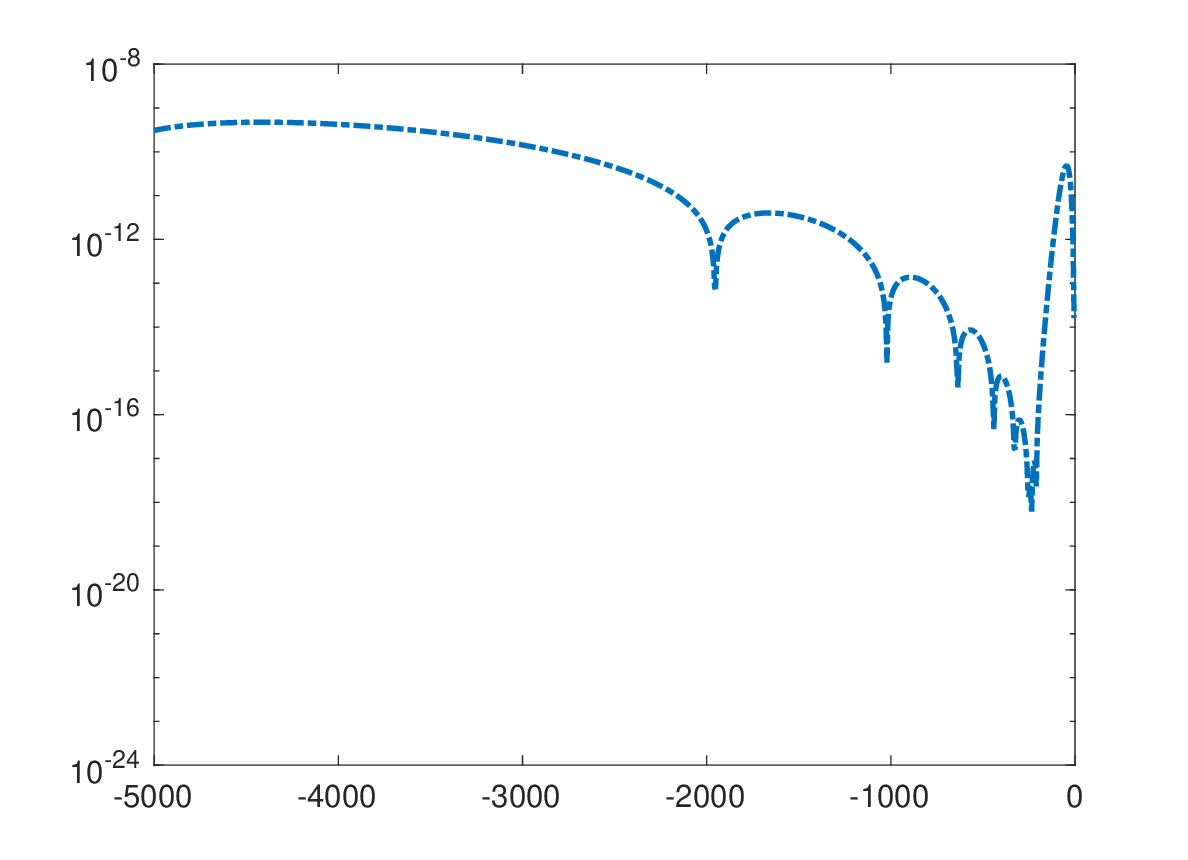}}
  \hfill
  \subfloat[$(p,\ell)=(2,3)$]{\includegraphics[width=0.5\textwidth]{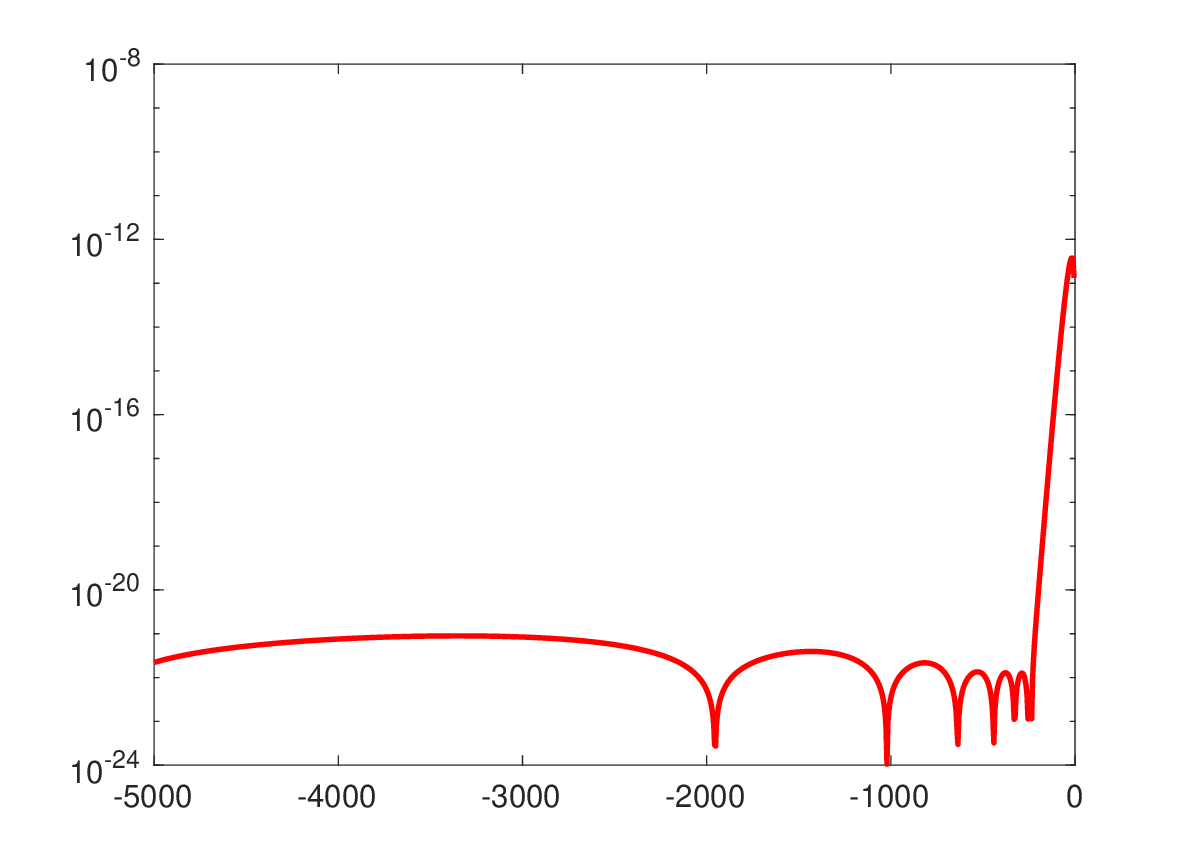}}
  \caption{Plots of  the absolute error $err_{p,N, \ell, r,\alpha}(w)$ with $N=100,$ $r=13$ and  $\alpha= 1/8.$}
\label{f3}
\end{figure}
As can be seen by comparing the pictures on the left with those on the right of the figure, the acceleration strategy proposed in this work ($\ell=3$ in this example) is more robust than the technique without acceleration ($\ell =0$) for large input values. To confirm this statement, in the next section we present the results of two numerical tests concerning the computation  of $q(\tau,A)\B f$ with $A\in \mathbb R^{s\times s}$ a matrix having a real spectrum located on the negative axis with many eigenvalues of large magnitude. Based on the above results, the use of the acceleration strategy can significantly influence the accuracy of the  computation of $q(\tau,w)$.

\section{Numerical Experiments}\label{numexp}
In this section, we present the results of some numerical experiments that we performed to compare the performance of our algorithm with the functional method described in \cite{Boito1,Boito} and the standard Arnoldi method.

The general problem to which we have referred is the differential problem of the form \eqref{l1}-\eqref{l2}. As stated in the introduction, the derivation of the numerical solution of such a BVP involves the computation of $q(\tau, A)\B f,$ where $\B f$ is the vector containing the evaluation of $f$ on the grid, cf. \eqref{l3}.

In our numerical experiments, we set $\B f={\tt ones}(s,1)$ and consider as exact solution (see \eqref{eq:init})
\[
\B z=({\tt expm}(A) -{\tt eye}(s))\textbackslash ({\tt expm}(\tau A) A \B f)
\]
derived using the MATLAB operator "backslash" and the built-in function ${\tt expm}$. 

After we have specified the approximate solution with $\B y \approx q(\tau, A)\B f$, we measure the error with
 \[
 \parallel \B y-\B z\parallel_\infty.
 \]
We now present four numerical tests. In the first two we compare our approach with the functional approximation of \cite{Boito1,Boito}, while in the last two we compare it with the Arnoldi iteration.

\subsection{Our approach compared to the Lanczos approximation} 
For the first two experiments, we focus on the non-local variant of the classical 1D heat equation, namely
\begin{equation}\label{lldiff1}
\left\{\begin{array}{ll}
u_t= u_{xx},&\quad 0<x<a,\; 0<t< t_f,\\
u(0,t)=u(a,t)=0, \end{array}
\right.
\end{equation}
with  the non-local condition 
\begin{equation}\label{lldiff2}
\frac1{t_f}\int_0^{t_f} u(x,t)\;dt=f(x).
\end{equation}

As is well known, the numerical solution of such a non-local BVP can be derived with a semidiscretization-in-space by applying the classical method of lines \cite{FW}. First, we discretize the one-dimensional operator
\[
\mathcal{A}= \displaystyle\frac{d^2}{dx^2}
\]
using central differences over the grid of points in the interval $[0,a]$ and obtain a matrix $A=(a_{i,j}) \in \mathbb R^{s\times s}.$ This then  reduces the problem \eqref{lldiff1}-\eqref{lldiff2} to a differential problem of the form \eqref{l1}-\eqref{l2}.

If we now use $\{x_i\}_{i=0}^{s+1},$ with $0=x_0<x_1< \dots <x_s<x_{s+1}$ $\leq$ $a,$ to denote the grid points on $[0,a],$ the matrix $A=(a_{i,j})$ has a tridiagonal structure with non-zero entries
\begin{equation}\label{mm}
\left\{
\begin{array}{lll}
a_{i,i}&=-2/((x_{i+1}-x_i)(x_i-x_{i-1})), &\ 1\leq i\leq s, \\
a_{i+1,i}&=\,\,\, 2/((x_{i+1}-x_{i})(x_{i+2}-x_{i})), &\ 1\leq i\leq s-1, \\
a_{i,i+1}&=\,\,\,  2/((x_{i+i}-x_{i})(x_{i+1}-x_{i-1})), &\ 1\leq i\leq s-1.
\end{array}
\right.
\end{equation}

When the grid is uniform with stepsize $h=a/(s+1),$ the matrix reduces to $A=(1/h^2) T$,   with $T$ the 1D Laplacian matrix. Note that small values of $h$ imply the presence of both large eigenvalues and a large norm of $A$. In light of the Theorem \ref{first}, the first is a critical fact for the convergence of the approximation \eqref{approxq}. The second, however, can deteriorate the accuracy of the acceleration strategy. 
To illustrate these effects, we performed numerical experiments in which we compared for $p=2n+2$ the approximation derived with \eqref{approxq} (for such $p$, \eqref{approxq} actually reduces to \eqref{Bern}) with that in \eqref{accapp}. For short, we denote by
\begin{itemize}
    \item[$\bullet$] {\bf{Lanc}}: $\parallel g_{n,N}(\tau,A)\B f-\B z\parallel_\infty$
    \item[$\bullet$]  {\bf{FastLanc}}: $\parallel \mathcal{G}_{p,N,\ell}(\tau,A)\B f-\B z\parallel_\infty$ 
\end{itemize}

For the tridiagonal matrix $A \in \mathbb R^{s\times s}$ given in \eqref{mm} with $s=512,$ our test suite consists of comparing the performances of these two approximations on two different grids. In particular, in $\mathbf{Test \, 1}$ we consider the uniform grid with $h=24/513$, while in $\mathbf{Test \, 2}$ we consider the not uniform grid with
 \begin{equation} \label{griglianoun}
 x_0=0, \quad x_1=0.01, \quad x_{i+1}=x_i+\sigma (x_i-x_{i-1}), \, \,i\geq 1, \, \sigma=1.005.
 \end{equation}
 In Table~\ref{t2} we give the errors of the two approximations on the uniform grid
 for $\tau=1/12, 1/6$ and $p=2$ (i.e. $n=0$), if the acceleration is taken into account. Analogous comparisons for the not uniform grid are shown in Table~\ref{t3}. These tables clearly show the robustness of the proposed acceleration scheme compared to the strategy based on using the Lanczos approximation. It should be emphasised that these results are obtained with essentially the same computational effort (see Remark~\ref{comput}).
\begin{table} 
\caption{$\mathbf{Test \, 1}$ - uniform grid}
\begin{center}{
\begin{tabular}{|c||c|c|c||c|c|c|}
\hhline{-||---||---}
 {\bf{Lanc}} & \multicolumn{3}{|c||}{$\tau=1/12$} & \multicolumn{3}{|c|}{$\tau=1/6$}  \\
\hhline{=::===::===}
\hhline{-||---||---}
{\backslashbox{$N$}{$n$}}
  &  2 & 3& 4 &  2 & 3& 4 \\
\hhline{=::===::===}
\multirow{1}{*}{50}
&$7.9e\!+05$ & $8.6e\!+08$ & $2.0e\!+14$ &$7.3e\!+05$ & $6.9e\!+08$ & $4.1e\!+12$ \\
\hhline{-||---||---}
\multirow{1}{*}{100}
 &$3.3e\!+04$ & $8.2e\!+06$ & $1.6e\!+12$&  $3.1e\!+03$ & $1.8e\!+06$ & $3.7e\!+12$ \\
\hhline{-||---||---}
\multirow{1}{*}{200}
 &$8.0e\!+02$ & $5.3e\!+05$ & $1.6e\!+12$ & $8.0e\!+02$ &$9.5e\!+05$ & $3.7e\!+12$ \\
 \hhline{-||---||---}
 \multicolumn{7}{c}{}  \\
  \multicolumn{7}{c}{}  \\
\hhline{-||---||---}
 {\bf{FastLanc}} & \multicolumn{3}{|c||}{$\tau=1/12$} & \multicolumn{3}{|c|}{$\tau=1/6$}  \\
\hhline{=::===::===}
{\backslashbox{$N$}{$\ell$}}
  &  2 & 3& 4 &  2 & 3& 4 \\
\hhline{=::===::===}
\multirow{1}{*}{50}
 &$1.3e\!-04$ & $7.1e\!-06$ & $4.9e\!-07$ &$7.2e\!-07$ & $6.7e\!-08$ & $1.3e\!-09$\\
\hhline{-||---||---}
\multirow{1}{*}{100}
 &$8.1e\!-06$ & $6.4e\!-08$ & $5.6e\!-10$ &$2.7e\!-07$ & $4.8e\!-11$ & $3.8e\!-12$\\
\hhline{-||---||---}
\multirow{1}{*}{200}
 &$1.8e\!-07$ & $6.9e\!-10$ & $3.8e\!-12$ &$4.8e\!-10$ & $6.0e\!-12$ & $3.8e\!-12$\\
\hhline{-||---||---}
\end{tabular}}
\label{t2}
\end{center}
\end{table}
\begin{table} 
\caption{$\mathbf{Test \, 2}$ - not uniform grid}
\begin{center}{
\begin{tabular}{|c||c|c|c||c|c|c|}
\hhline{-||---||---}
 {\bf{Lanc}} & \multicolumn{3}{|c||}{$\tau=1/12$} & \multicolumn{3}{|c|}{$\tau=1/6$}  \\
\hhline{=::===::===}
\hhline{-||---||---}
{\backslashbox{$N$}{$n$}}
  &  2 & 3& 4 &  2 & 3& 4 \\
\hhline{=::===::===}
\multirow{1}{*}{50}
&$3.4e+12$ & $1.7e+18$ & $7.2e+23$  &$3.6e+12$ & $1.5e+18$ & $1.8e+24$ \\
\hhline{-||---||---}
\multirow{1}{*}{100}
&$1.6e+11$ & $1.9e+16$ & $2.1e+23$ &$6.5e+09$ & $2.5e+15$ & $1.0e+24$\\
\hhline{-||---||---}
\multirow{1}{*}{200}
 &$2.6e+09$ & $7.7e+14$ & $2.1e+23$ &$3.8e+09$ & $1.7e+15$ & $1.9e+24$\\
 \hhline{-||---||---}
 \multicolumn{7}{c}{}  \\
  \multicolumn{7}{c}{}  \\
\hhline{-||---||---}
 {\bf{FastLanc}} & \multicolumn{3}{|c||}{$\tau=1/12$} & \multicolumn{3}{|c|}{$\tau=1/6$}  \\
\hhline{=::===::===}
{\backslashbox{$N$}{$\ell$}}
  &  2 & 3& 4 &  2 & 3& 4 \\
\hhline{=::===::===}
\multirow{1}{*}{50}
 &$2.8e\!-03$ & $1.5e\!-04$ & $1.0e\!-05$ &$1.5e\!-05$ & $1.4e\!-06$ & $2.7e\!-08$\\
\hhline{-||---||---}
\multirow{1}{*}{100}
 &$1.7e\!-04$ & $1.4e\!-06$ & $1.3e\!-08$&$5.9e\!-06$ & $1.0e\!-09$ & $8.5e\!-11$\\
\hhline{-||---||---}
\multirow{1}{*}{200}
   &$4.1e\!-06$ & $1.5e\!-08$ & $1.4e\!-10$&$4.8e\!-09$ & $1.3e\!-10$ & $8.5e\!-11$\\
\hhline{-||---||---}
\end{tabular}}
\label{t3}
\end{center}
\end{table}

\subsection{Our approach compared to the Arnoldi approximation} 
We now consider the problem \eqref{l1}, \eqref{l2}, which corresponds to two different matrices. In $\mathbf{Test \, 3}$, the matrix $A$ is the one associated with the not uniform grid given in \eqref{griglianoun}. This matrix has a wide spectrum. All its eigenvalues are negative and lie in the interval $[-3.7 \cdot 10^4, -1.7 \cdot 10^{-2}].$ In $\mathbf{Test \, 4}$ we consider the scaled-unitary matrix
 \[
A=10^{-8}\begin{bmatrix*}
 0& \ldots &\ldots& 0& 1 \\
 1 &0 &\ldots& \ldots & 0\\
0 &\ddots &\ddots & & \vdots\\
 \vdots&\ddots & \ddots & \ddots & \vdots\\
 0& \dots & 0& 1& 0
 \end{bmatrix*} := 10^{-8} C,
 \]
 where $C$ is the generator of the circulant matrix algebra. All its eigenvalues lie on a small circle centered at the origin of the complex plane and of radius $10^{-8}$.

Then, for $\tau=1/6,$ we approximate $q(\tau, A)\B f$ by
 \begin{itemize}
 \item[$\bullet$] $\mathcal{G}_{p,N,\ell }(\tau, A)\B f$;
 \item[$\bullet$] the Arnoldi iteration at step $j$:
 \[
\B y_j= V_j \left ({\tt expm}(H_j) -{\tt eye}(j) \right)\textbackslash \left({\tt expm}(\tau H_j) H_j \B e_1 \parallel \B f\parallel_2 \right).
\]
We recall that the Arnoldi method consists in projecting $A$ onto a Krylov subspace of small dimension. Starting from a unit norm vector, the Arnoldi method builds the matrices $V_j$ incrementally. The projected part $H_j=V_j^TAV_j$ of $A$ is then used to evaluate $q(\tau, H_j)$ and obtain an approximation of $q(\tau,A)\B f$ (for more details see e.g. \cite{Saad1} and the references given therein).
 \end{itemize}
 
In Figure~\ref{farnoldi} we show the convergence history and a measure for the loss of orthogonality of the Arnoldi method, applied to the matrices described above.
\begin{figure}
 \subfloat[$\mathbf{Test \, 3}$]{\includegraphics[width=0.5\textwidth]{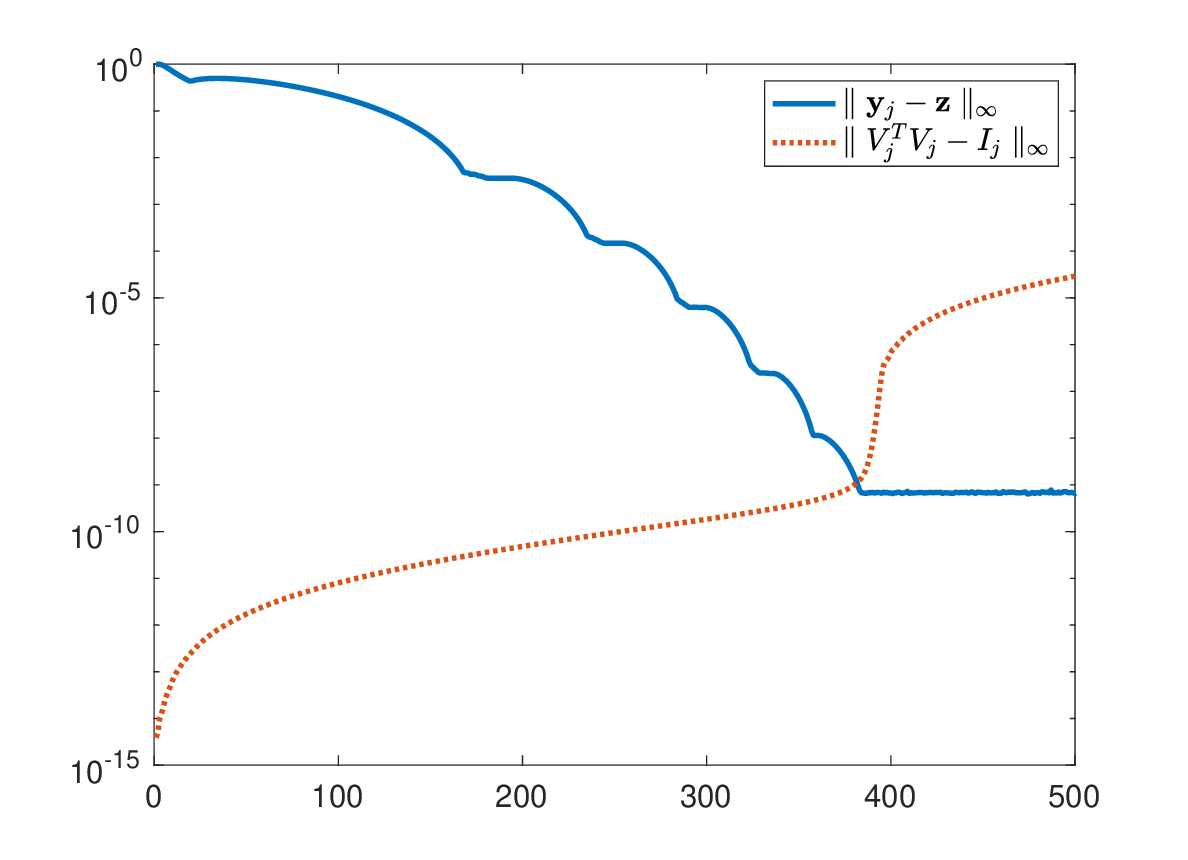}}
 \hfill
 \subfloat[$\mathbf{Test \, 4}$]{\includegraphics[width=0.5\textwidth]{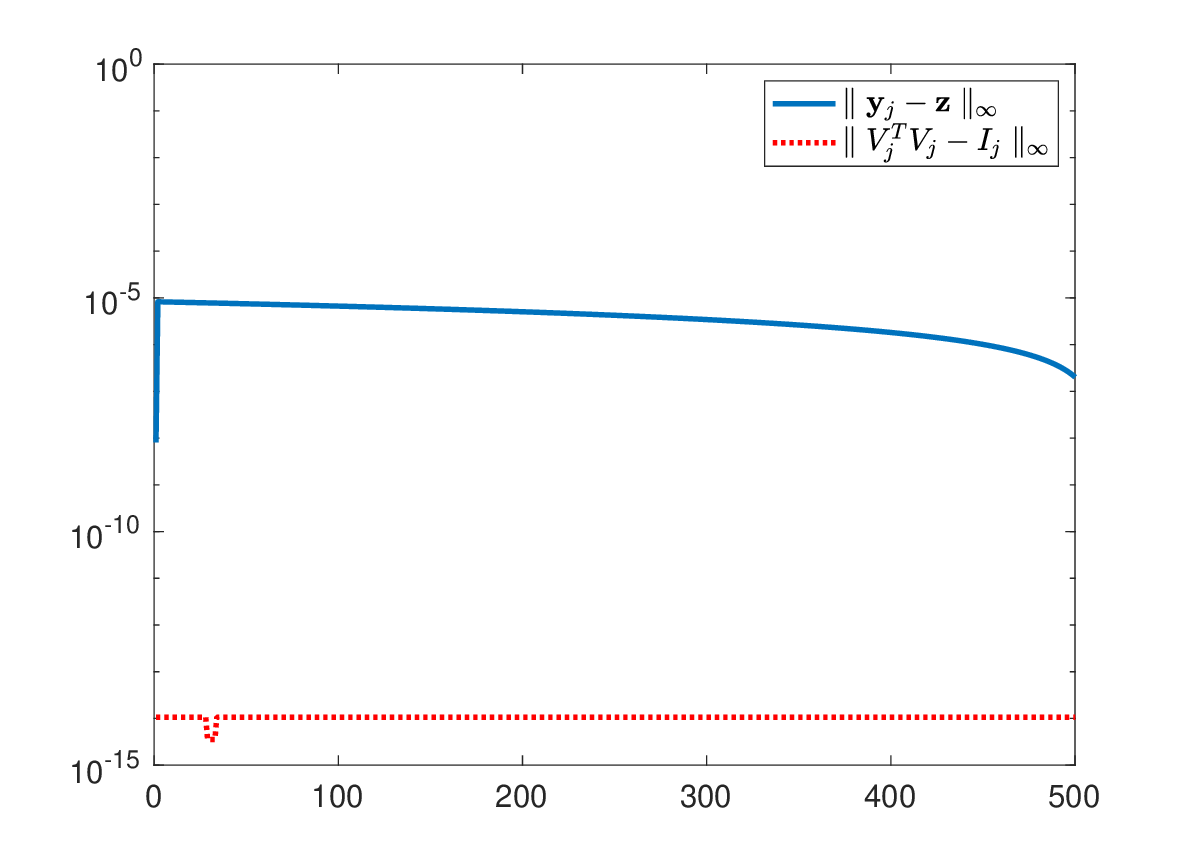}}
 \caption{Plots of the error and orthogonality (loss) in the Arnoldi method vs the iteration at step $j$.}
\label{farnoldi}
\end{figure}
The speed of convergence seems to be quite slow due to the wide range of the spectrum or due to conditioning problems related to intermediate calculations. In contrast, our method works quite well. Indeed, for $\mathbf{Test \, 3}$, 
 the computation of $\mathcal{G}_{2,50,5}(\tau,A)\B f$ takes $0.35$ seconds  and we find that
\[
\parallel \mathcal{G}_{2,50,5}(\tau,A)\B f-\B z\parallel_\infty=1.3 \cdot 10^{-10},
\]
while to achieve the same accuracy the evaluation of the vectors
$\B y_j$ 
is performed in 
about 8.5 seconds. \\ Although the matrix in $\mathbf{Test \, 4}$ can be problematic for the Arnoldi method, as we can see in the right picture of Figure~\ref{farnoldi}, since for our approach $q(\tau,w)$ has a removable singularity at $w=0$, our approach performs well and we actually have
\[
\parallel \mathcal{G}_{2,50,4}(\tau,A)\B f-\B z\parallel_\infty=2.7 \cdot 10^{-17}.
\]

\section{Conclusions and Future Work}\label{end}
In this paper we have derived the series expansion of the generating function of Bernoulli polynomials in the framework of the so-called Lanczos acceleration method of Fourier-trigonometric series of smooth non-periodic functions. This is a general method that involves the use of Bernoulli polynomials to accelerate convergence. Our derivation makes it possible to adapt the general error estimates provided for this method to the meromorphic function under consideration, namely the generating function of Bernoulli polynomials. Extensions of this method involving the use of rational functions are also investigated. In particular, we propose a cost-effective acceleration method based on the rational-trigonometric approximation of the residual of the Fourier series. Numerical experiments with finite difference discretizations of differential operators show that our proposed acceleration method outperforms the Lanczos 
approximation in terms of robustness and accuracy. Future work will focus on the development of an automatic procedure for the selection of the parameters necessary for the construction of the resulting approximations. In this respect, it would be useful to derive an error estimate for the residual $S_{p,N, \ell}(\tau, w)$.  Another interesting topic is the analysis of Fourier-Pad\'e approximation methods of the residual. Such methods can be very effective when combined with efficient scaling and squaring techniques. The development of these methods for the calculation of $q(\tau,w)$ is still an ongoing research project.
 \begin{acknowledgements}
The authors are members of GNCS-INDAM. Lidia Aceto is partially supported by the "INdAM - GNCS Project", codice CUP$\_$E53C23001670001. Luca Gemignani is  partially supported by  European Union - NextGenerationEU under the National Recovery and Resilience Plan (PNRR) - Mission 4 Education and research - Component 2 From research to business - Investment 1.1 Notice Prin 2022 - DD N. 104  2/2/2022, titled Low-rank Structures and Numerical Methods in Matrix and Tensor Computations and their Application, proposal code 20227PCCKZ – CUP I53D23002280006  and by the Spoke 1 ``FutureHPC \& BigData”  of the Italian Research Center on High-Performance Computing, Big Data and Quantum Computing (ICSC)  funded by MUR Missione 4 Componente 2 Investimento 1.4: Potenziamento strutture di ricerca e creazione di "campioni nazionali di R\&S (M4C2-19 )" - Next Generation EU (NGEU).
\end{acknowledgements}

\end{document}